\documentclass[11pt]{article}
\usepackage[utf8]{inputenc}
\usepackage{authblk}
\usepackage{amsfonts,amssymb,amsthm,amscd,amsmath,latexsym,amsbsy,enumerate,a4wide, tikz-cd,
tikz,  ifthen, mathrsfs, mathtools}
\usepackage{multicol}

\def\C{\mathbb{C}}
\def\N{\mathbb{N}}
\def\D{\mathcal{D}}

\theoremstyle{definition}
\newtheorem{theorem}{Theorem}
\newtheorem{prop}[theorem]{Proposition}
\newtheorem{rmk}[theorem]{Remark}
\newtheorem{lem}[theorem]{Lemma}
\newtheorem{corollary}[theorem]{Corollary}
\newtheorem{definition}[theorem]{Definition}

\theoremstyle{remark}

\numberwithin{equation}{section}

\begin{document}

\title{Non Abelian Toda-type equations and matrix valued orthogonal polynomials}
\author[1]{Alfredo Deaño\thanks{alfredo.deanho@uc3m.es}}
\affil[1]{Department of Mathematics, Universidad Carlos III de Madrid, Spain}
\author[2]{Lucía Morey\thanks{lmorey@unc.edu.ar, lucia.morey@gtiit.edu.cn}}
\affil[2]{FaMAF-CIEM, Universidad Nacional de C\'ordoba (Argentina)\\ Guangdong Technion Israel Institute of Technology (China)}
\author[3]{Pablo Román\thanks{pablo.roman@unc.edu.ar, pablo.roman@gtiit.edu.cn}}
\affil[3]{FaMAF-CIEM, Universidad Nacional de C\'ordoba (Argentina)\\Guangdong Technion Israel Institute of Technology (China)}

\date{\today}

\maketitle
\begin{abstract}
In this paper, we study parameter deformations of matrix valued orthogonal polynomials (MVOPs). These deformations are built on the use of certain matrix valued operators which are symmetric with respect to the matrix valued inner product defined by the orthogonality weight. We show that the recurrence coefficients associated with these operators satisfy generalizations of the non-Abelian lattice equations. We provide a Lax pair formulation for these equations, and an example of deformed Hermite-type matrix valued polynomials is discussed in detail.
\end{abstract}

\maketitle

\section{Introduction}

The Toda lattice is a system of differential equations introduced by Morikazu Toda \cite{toda1967vibration,toda1967wave} to describe the evolution of a system of $N$ particles on the real line with an exponential interaction: namely, the equations of motion are
\begin{equation}
\ddot{q}_n=e^{q_{n-1}-q_n}-e^{q_{n}-q_{n+1}}, \qquad 1\leq n\leq N,
\end{equation}
where $q_n=q_n(t)$ is the displacement of the $n$-th particle, and we use the standard notation $\dot f = \frac{df}{dt}$. An alternative formulation, that shows the complete integrability of the system, is due to Flaschka \cite{FlaschkaToda} and to Manakov \cite{ManakovToda}: by defining the new variables 
\begin{equation}
a_n^2=e^{q_{n-1}-q_{n}}, \qquad
b_n=-p_{n-1}
\end{equation}
where $p_n=\dot{q}_n$ is the momentum of the $n$-th particle, the equations of motion of the Toda lattice can be written as 
\begin{equation}\label{eq:TodaFlaschka}
\frac{d}{dt}a_n^2=a_n^2(b_{n-1}-b_{n}),\qquad 
\frac{d}{dt}b_n=a_{n}^2-a_{n+1}^2.
\end{equation}
If we take the semi-infinite case, with boundary condition $q_{-1}=-\infty$, 
then $a^2_0=0$ and the system can be written in the form of a Lax pair $\dot{J}=\left[J,B\right]$, where $J$ is a symmetric tridiagonal operator and $B$ is a skew symmetric bidiagonal operator:
\begin{equation}
J
=
\begin{pmatrix}
b_0 & a_1 & 0 & & \\
a_1 & b_1 & a_2 & \ddots  \\
0 & a_2 & b_2 & \ddots   \\
& \ddots & \ddots & \ddots \\
\end{pmatrix},\qquad
B
=
\frac{1}{2}\begin{pmatrix}
0 & a_1 & 0 & &  \\
-a_1 & 0 & a_2 & \ddots &\\
0 & -a_2 & 0 & \ddots &   \\
&  \ddots & \ddots & \ddots  \\
\end{pmatrix}.
\end{equation}

Because of this structure, the Toda lattice equations arise in a natural way in the theory of orthogonal polynomials: given a positive measure $\mu$ with finite moments on $\mathbb{R}$, one can construct a family of orthonormal polynomials $p_n(x)$ that satisfy
\[
\int_{\mathbb{R}} p_n(x) p_m(x) d\mu(x)=\delta_{n,m}
\]
for $n,m=0,1,\ldots$, as well as a three term recurrence relation
\begin{equation}\label{eq:TTRR}
xp_n(x)=a_{n+1}p_{n+1}(x)+b_np_n(x)+a_np_{n-1}(x),
\end{equation}
with initial values $p_{-1}(x)=0$ and $p_0(x)=\mu_0^{-1/2}$, where $\mu_0=\int_{\mathbb{R}} d\mu(x)$. Conversely, if we have a recurrence relation like \eqref{eq:TTRR} with $a_n>0$, Favard's theorem (see for instance \cite[Theorem 4.4]{Chihara} or \cite[Theorem 2.5.2]{Ismail}) guarantees existence of a measure $d\mu(x)$ with respect to which there exists a family of orthonormal polynomials that satisfies such a recurrence relation. A similar result holds if we introduce a deformation of the coefficients $a_n,b_n$ (and therefore of the measure of orthogonality, the orthonormal polynomials and all related quantities) with respect to a time parameter $t$. More precisely, if we take the deformed measure $e^{-tx}d\mu(x)$, it is direct to check that the deformation of the recurrence coefficients $a_n(t),b_n(t)$ with respect to $t$ coincides with the Toda lattice equations \eqref{eq:TodaFlaschka}. Other deformations of a similar kind, such as $e^{-tx^d}d\mu(x)$, with $d\geq 2$, lead to other lattice equations that are relevant as well. Multi-time deformations of the measure, such as $e^{\sum_{i=1}^{\infty} t_ix_i}d\mu(x)$, are of importance as well, leading to the discrete KP hierarchy \cite{AvM99}.

We note that instead of orthonormal polynomials $p_n(x)$, it is common to use monic ones $P_n(x)$, and we can write $p_n(x;t)=\gamma_n(t) P_n(x;t)$. In that case, the recurrence relation reads 
\begin{equation}\label{eq:TTRR_monic}
xP_n(x;t)=P_{n+1}(x;t)+b_n(t) P_n(x;t)+a_n^2(t) P_{n-1}(x;t).
\end{equation}
The corresponding Jacobi matrix is no longer symmetric, but we can make the transformation
\[
\widetilde{J}(t)=\Gamma^{-1}(t)\,J(t)\,\Gamma(t), \qquad \Gamma(t)=\textrm{diag}(\gamma_0(t),\gamma_1(t), \ldots).
\]
Then the Lax pair has the form $\dot{\widetilde{J}}=\left[\widetilde{J},\widetilde{B}\right]$, with 
\begin{equation}\label{eq:JBtilde}
\widetilde{J}
=
\begin{pmatrix}
b_0 & 1 & 0 & & \\
a_1^2 & b_1 & 1 & \ddots &\\
0 & \ddots & \ddots & \ddots \\
& \ddots & \ddots & \ddots \\
\end{pmatrix},\qquad
\widetilde{B}
=
\begin{pmatrix}
b_0 & 1 & 0 & & \\
0 & b_1 & 1 & \ddots &\\
& \ddots & \ddots & \ddots \\
& & \ddots & \ddots \\
\end{pmatrix}.
\end{equation}
We have included this formulation with monic orthogonal polynomials explicitly because it is analogous to the result that we present for matrix valued orthogonal polynomials (MVOPs) in Section \ref{subsec:Lax}.

A particularly important connection can be established between orthogonal polynomials and Painlev\'e equations, both differential and discrete, we refer the reader to \cite{Walter_book} for a recent account. In this context, Painlev\'e equations for recurrence coefficients arise as a result of compatibility between differential equations, recurrence relations and time evolution equations in $t$ satisfied by orthogonal polynomials, with a suitable $t$ deformation of the orthogonality measure.

The connection between integrable systems and matrix orthogonal polynomials has been explored in the literature, mainly in the direction of obtaining non-Abelian versions of classical objects such as Painlev\'e equations, see for example \cite{CafassoI,CafassoII}, for Hermite and Laguerre matrix polynomials, and also \cite{CCM2019,BFM_JMAA}. We also mention the recent work \cite{Arniz-2014} as well as \cite{BFGA_2020} on matrix Toda and Volterra equations, as well as a more general analysis of ladder operators for matrix valued orthogonal polynomials in \cite{DER}, based on ideas by Casper and Yakimov in \cite{Casper2}. The non-Abelian Toda lattice is a non-commutative analogue of the classical Toda lattice studied by Bruschi et al. \cite{Bruschi} and Gekhtman \cite{GeK}. One can obtain solutions of the non-Abelian Toda lattice by considering the coefficients of the three-term recurrence relation for matrix valued orthogonal polynomials where the matrix weight is deformed by an exponential factor. In this paper, we are interested in exploring further non-Abelian Toda-type identities for matrix valued deformations of matrix valued orthogonal polynomials.

We consider an $N\times N$ positive definite Hermitian matrix weight function $W(x)$ on the real line with trivial kernel (a.e) and a positive measure $\mu$
on the interval $[a,b]$, where $a$ and $b$ are allowed to be $\pm \infty$. Assuming that all the moments $\mu_m=\int_{\mathbb{R}} |x|^m W(x)dx$, $m=0,1,\ldots$,  are finite, $W$ defines a matrix valued inner product 
$\langle \cdot, \cdot \rangle$ on the space of matrix valued polynomials $M_N(\C)[x]$ by
\begin{equation}\label{eq:orth_general}
\langle P, Q \rangle = \int_a^b P(x) W(x) Q(x)^\ast d\mu(x),
\end{equation}
 where $^\ast$ denotes the conjugate transpose. 
It is well known that for such a weight there exists a unique sequence of monic matrix valued polynomials (MVOPs) $(P_n)_{n\geq 0}$ such that
\begin{equation}\label{eq:MVOPs}
    \langle P_n(x),P_m(x)\rangle
    =
    \int_{\mathbb{R}}
    P_n(x)W(x)P_m(x)^\ast d\mu(x)
    =
    \mathcal{H}_n\delta_{n,m},
\end{equation} 
where $\mathcal{H}_n$ is a positive definite matrix and $\delta_{n,m}$ is the standard Kronecker delta, see for instance \cite{Berg,DPS}. The sequence of monic matrix valued orthogonal polynomials $(P_n)_{n\geq 0}$ satisfies a three-term recurrence relation
\begin{equation}
\label{eq:three-term}
xP_n(x) = P_{n+1}(x)+B(n)P_n(x) +C(n) P_{n-1}(x),
\end{equation}
where the coefficients $B(n)$ and $C(n)$ are $N\times N$ matrices.

The classical Toda deformation for $W$ is analogous to the scalar case, and it is given by
$$W(x;t) = e^{-tx} W(x), \qquad t\in \mathbb{R},$$
see e.g. \cite{Miranian, IKR2}. The main goal of this paper is to study a new class of deformations 
of a weight matrix $W(x)$ by replacing the role of $x$ in the exponential term $e^{-tx}$ 
by a matrix polynomial $\Lambda(x)$ with a symmetry property with respect to the weight. This construction includes deformation with any scalar polynomial with real coefficients, which leads to the Toda hierarchy.


The paper is organized as follows: In Section \ref{sec:Fourier} we introduce the so called Fourier algebras of differential and difference operators associated to the weight $W$. In Section \ref{sec:deformations} we present the main results of the paper: Section \ref{subsec:sym} introduces deformations of the weight of the form $e^{-t\Lambda(x)}W(x)$, where $t\in\mathbb{R}$ and $\Lambda(x)$ satisfies a suitable symmetry condition with respect to $W(x)$; given this differential operator $\Lambda(x)$, there exists a corresponding difference operator $M(t)$, with coefficients $G_j(n;t)$, whose time evolution equations we study in Section \ref{subsec:Toda}; these equations can be written in the form of a Lax pair, involving tridiagonal and bidiagonal block operators, as shown in Section \ref{subsec:Lax}. In Section \ref{sec:polToda}, for a given $\Lambda(x)$, we study deformations of the form $e^{-v(\Lambda(x);t)}W(x)$, where $v(x;t)$ is any polynomial of degree $k$ with positive leading coefficient. We conclude the paper with a detailed example of deformed Hermite-type polynomials in Section \ref{sec:Hermite}.
 
\section{Fourier algebras of differential and difference operators}\label{sec:Fourier}

The three-term recurrence relation \eqref{eq:three-term} links the action of two operators on the sequence of matrix orthogonal polynomials $(P_n)_{n\geq 0}$. On the left hand side we have the operator multiplication by $x$, which is a symmetric operator with respect to the matrix inner product \eqref{eq:orth_general}, since $\langle xP , Q \rangle = \langle P, xQ\rangle$ for arbitrary matrix polynomials $P,Q$. As a differential operator (of order $0$), $x$ acts on $P_n$ from the right
$$P_n(x) \cdot x = xP_n(x).$$
The right hand side of \eqref{eq:three-term} can be interpreted as the action of a difference operator $\mathcal{L}$ on the sequence $(P_n)_{n\geq 0}$. More precisely, if we denote by $\delta^j$ the shift operator $\delta^j\cdot P_n(x) = P_{n+j}(x)$, we can rewrite \eqref{eq:three-term} as
\begin{equation}
\label{eq:three_term_as_opearators}
P_n (x) \cdot x  = \mathcal{L} \cdot P_n(x),\qquad \mathcal{L} = \delta + B(n) \delta^0 + C(n)\delta^{-1}.
\end{equation}
More generally, we consider a differential operator $\mathcal{D}$ which acts on the variable $x$ of $P_n(x)$ by
\begin{equation}
\label{eq:DifferentialOperator}
\D=\sum_{j=0}^n \partial_x^j F_j(x), \qquad P_n(x) \cdot \D  = \sum_{j=0}^n \left(\partial_x^j P_n(x)\right)\,  F_j(x), \qquad \partial_x^j := \frac{d^j}{dx^j},
\end{equation}
The matrix coefficients $F_j:\mathbb{C}\to M_N(\mathbb{C})$ are assumed to be rational functions of $x$. We observe that the matrix coefficients $F_j(x)$ multiply the polynomials from the right. Similarly, we consider a left action on the sequence $(P_n)_{n\geq 0}$ by difference operators of the form
\begin{equation}
\label{eq:DifferenceOperator}
M=\sum_{j=-\ell}^k G_j(n) \delta^j, \qquad M \cdot P_n(x) = \sum_{j=-\ell}^k G_j(n) \, \delta^j  P_n(x) = \sum_{j=-\ell}^k G_j(n) \, P_{n+j}(x).
\end{equation}
Here, $G_j:\mathbb{N}_0 \to \mathrm{M}_n(\mathbb{C})$ for all $j=-\ell,\ldots,k$, and $P_n(x)=0$ for all $n<0.$ We denote by $\mathcal{M}_N$ the algebra of all differential operators of the form \eqref{eq:DifferentialOperator} and by $\mathcal{N}_N$ the algebra of all difference operators of the form
\eqref{eq:DifferenceOperator}. 

Following \cite{Casper2}, we have the so called Fourier algebras of operators associated to $(P_n)_{n\geq 0}$:
\begin{equation}
\label{eq:definition-Fourier-algebras}
\begin{split}
\mathcal{F}_L(P)&=\{ M\in \mathcal{N}_N \colon \exists \D\in \mathcal{M}_N,\, M\cdot P = P\cdot \D \} \subset \mathcal{N}_{N},\\
 \mathcal{F}_R(P)&=\{ \D\in \mathcal{M}_N \colon \exists M\in \mathcal{N}_N,\, M\cdot P = P\cdot \D \}\subset \mathcal{M}_{N}.
\end{split}
\end{equation}
Observe that the relation \eqref{eq:three_term_as_opearators} implies that for any weight matrix $W$ and the corresponding monic orthogonal polynomials $(P_n)_{n\geq 0}$, we have that $x\in \mathcal{F}_R(P)$ and $\mathcal{L}\in \mathcal{F}_L(P)$. The two Fourier algebras $\mathcal{F}_R(P)$ and $\mathcal{F}_L(P)$ are isomorphic via the  generalized Fourier map
\begin{equation}
\varphi: \mathcal{F}_L(P)\mapsto \mathcal{F}_R(P),
\end{equation}
given by $\varphi(M)=D$, whenever $P_n\cdot D = M\cdot P_n$. 

For a differential operator $\mathcal{D}\in \mathcal{M}_N$, we say that $\mathcal{D}^\dagger$ is an adjoint of $\mathcal{D}$ in $M_N(\C)[x]$ if
$$\langle P\cdot \mathcal{D} , Q \rangle = \langle P, Q\cdot \mathcal{D}^\dagger\rangle,\qquad \text{ for all }P,Q \in M_N(\C)[x].$$
Similarly, for a difference operator $M\in \mathcal{N}_N$, we say that $M^\dagger$ is an adjoint of $M$ in $M_N(\C)[x]$ if
$$\langle M\cdot P_n, P_m \rangle = \langle P_n, M^\dagger \cdot P_m \rangle,\qquad \text{ for all }n,m \in \N_0.$$
The operators $\mathcal{D}$ and $M$ are symmetric if $\mathcal{D}=\mathcal{D}^\dagger$ and 
$M=M^\dagger$, respectively. As proved in \cite[Corollary 3.8]{Casper2}, the left and right Fourier algebras are closed under the adjoint $\dagger$.

The main property of the matrix valued inner product \eqref{eq:orth_general} is the fact that the operator multiplication by $x$ is a symmetric operator: $\langle P\cdot x,Q\rangle = \langle P, Q\cdot x\rangle$ for all matrix valued polynomials $P$ and $Q$. The three-term recurrence relation \eqref{eq:three_term_as_opearators} for the monic matrix valued orthogonal polynomials is a consequence of this property. This property is also essential in the study of the time evolution of the recurrence coefficients of deformations of the weight matrix by a factor of the form $e^{v(x;t)}$, where $v$ is a polynomial of $x$ depending on the time $t$. This setting was considered in \cite[Section 7]{DER}, see also \cite{IKR2}.

In this paper, we will be concerned with operators $\Lambda \in \mathcal{M}_N$ of order zero which are symmetric with respect to the inner product \eqref{eq:orth_general}, this is
$$\langle P\cdot \Lambda(x) , Q \rangle = \langle P, Q\cdot \Lambda(x) \rangle,$$
for all $P,Q\in M_N(\mathbb{C})[x]$. Symmetric operators of order zero are characterized as follows:
\begin{definition}
\label{def:1}
We denote by $\mathcal{S}(W)$ the set of all symmetric operators of order zero:
\begin{equation}\label{eq:defS}
\mathcal{S}(W) :=
\{\Lambda(x) \in \mathcal{M}_N: \Lambda(x) W(x) = W(x)\Lambda(x)^\ast ,\quad \forall x\in[a,b]\}.
\end{equation}
\end{definition}
According to \cite[Theorem 3.7]{Casper2} the symmetry condition implies that $\Lambda\in \mathcal{F}_R(P)$, and therefore, there exists a corresponding difference operator $\varphi^{-1}(\Lambda) \in \mathcal{F}_L(P)$. Moreover if we set $n=0$ in the relation
$$P_n(x)\cdot \Lambda(x)=\varphi^{-1}(\Lambda) \cdot P_n(x), \qquad n\in \N_0,$$
we obtain that $\Lambda$ is a matrix valued polynomial. It follows that $\mathcal{S}(W) \subseteq \mathcal{F}_R(P)$.

Constant symmetric operators of order zero lead to reducible weight matrices, see for instance \cite{KR,TiraoZurrian}, and for this reason it is of great interest to determine whether there exist such operators for a given 
matrix weight. In general, this is not an easy task, see for instance \cite{KR,KdlRR,IKR2,KoelinkRlaguerre}.

In this paper we deal with symmetric operators of order zero and higher degree (i.e. not constant), whose existence does not say anything about the reducibility of the weight matrix, see for instance \cite{DER,EMR} for non-trivial examples.

\section{Matrix valued deformations of the weight}\label{sec:deformations}
\subsection{Symmetry conditions}\label{subsec:sym}
The goal of section is to introduce a class of deformations of a weight matrix $W(x)$ by replacing the role of $x$ with a symmetric operator of order zero $\Lambda(x)\in \mathcal{S}(W)$. We consider a deformation of the form
\begin{equation}\label{eq:WLambda}
W(x;t)
=
e^{-t \Lambda(x)}W(x), \qquad t\in\mathbb{R}.
\end{equation}
We can verify that the deformed weight is Hermitian:
\[
W(x;t)^\ast
=
W(x)^\ast e^{-t \Lambda(x)^\ast}
=
W(x) e^{-t \Lambda(x)^\ast}
=
e^{-t \Lambda(x)} W(x) 
=
W(x;t).
\]
This property also follows from Definition \ref{def:1}, since the weight matrix $W(x;t)$ can be written in the symmetric form
$$
W(x;t)
=
e^{-\frac{t}{2} \Lambda(x)}W(x) 
e^{-\frac{t}{2} \Lambda(x)^\ast}
.$$
As for the classical Toda deformation, the main goal is to obtain evolution equations for relevant quantities related to MVOPs with respect to $t$. The symmetry condition in Definition \ref{def:1} is a key property that will allow us to move the $\Lambda(x)$ factor when differentiating with respect to $t$ and integrating by parts.

 Assuming that all moments of the deformed weight are finite, we construct a sequence of monic orthogonal polynomials $(P_n(x;t))_{n\geq 0}$ that satisfy
\begin{equation}\label{eq:MVOPs_t}
    \langle P_n(x;t),P_m(x;t)\rangle_t
    :=
    \int_{\mathbb{R}}
    P_n(x;t)W(x;t)P_m(x;t)^\ast d\mu(x)
    =
    \mathcal{H}_n(t)\delta_{n,m},
\end{equation} 
where $\mathcal{H}_n(t)$ is an invertible positive definite matrix. Associated to the weight matrix $W(x;t)$, we have the deformed left and right Fourier algebras $\mathcal{F}_L(P;t)$ and $\mathcal{F}_R(P;t)$. We observe that
\begin{equation}
\label{eq:lamda-intersection}
\Lambda(x)W(x;t)
=
\Lambda(x)e^{-t\Lambda(x)}W(x)
=
e^{-t\Lambda(x)}\Lambda(x)W(x)
=
e^{-t\Lambda(x)}W(x)\Lambda(x)^\ast
=
W(x;t)\Lambda(x)^\ast.
\end{equation}
This implies that $\Lambda(x)$ is in the intersection of all the $\mathcal{F}_R(P;t)$, for $t\geq 0$. Moreover, we observe that 
for any polynomial $v$ with real coefficients, $v(\Lambda(x))\in S(W(x;t))$ and therefore $v(\Lambda(x))\in \mathcal{F}_R(P;t)$. 

As a consequence, for every $t>0$ there exists $M(t)=\varphi^{-1}_t(\Lambda(x))\in\mathcal{F}_L(P;t)$, such that 
\begin{equation}
\label{eq:actionLambdaPn}
P_n(x;t)\cdot \Lambda(x)=M(t) \cdot P_n(x;t), \qquad n\in \N_0.
\end{equation}
\begin{rmk}
As in the previous section, \eqref{eq:actionLambdaPn} implies that $\Lambda(x)$ is a matrix valued polynomial. If
\begin{equation}\label{eq:Lambdak}
\Lambda(x)=\Lambda_k x^k+\Lambda_{k-1}x^{k-1}+\cdots+\Lambda_0,
\end{equation}
then by \cite[Definition 3.4]{Casper2} and the fact that $\Lambda(x)=\Lambda(x)^\dagger$ and $M(t)=M(t)^\dagger$, the difference operator $M(t)$ has $2k+1$ terms
\begin{equation}
\label{eq:M-2k+1}
M(t)=\sum_{j=-k}^k G_j(n;t) \delta^j.
\end{equation}
\end{rmk}
For the rest of the paper we will be concerned with the time evolution of the coefficients $G_j(n;t)$ of the operator \eqref{eq:M-2k+1}.

\begin{lem}
\label{lem:3}
The coefficient $G_k$ of $M(t)$ is independent of $t$ and $n$.
\end{lem}
\begin{proof}
The proof follows by comparing the leading coefficients of $P_n(x;t)\cdot \Lambda(x)$ and $M(t)\cdot P_n(x;t)$.
Since $P_n(x;t)\cdot \Lambda(x)=M(t)\cdot P_n(x;t)$ for all $t\geq0,$ and 
$$P_n(x;t)\cdot \Lambda(x)=x^{n+k}\Lambda_k + \text{lower order terms},$$
$$M(t)\cdot P_n(x;t)=G_k(n;t)x^{n+k}+\text{lower order terms},$$
by comparing the coefficient of $x^{n+k}$ we get $G_k(n;t)=\Lambda_k,$ and therefore $G_k(n;t)$ is independent of $t$ and $n$. 
\end{proof}
\begin{prop}
The following relation holds true:
\begin{equation}\label{eq:WeakPearson-H}
\begin{aligned}
    G_{\ell}(n;t)\mathcal{H}_{n+\ell}(t)&=\mathcal{H}_n(t)G_{-\ell}(n+\ell;t)^\ast,\qquad \ell=-k,\ldots,k.
\end{aligned}
\end{equation}
\end{prop}
\begin{proof}
For $\ell\in\{-k,\ldots,k\}$, we have

$$G_\ell(n;t)P_{n+\ell}(x;t)=M(t)\cdot P_n(x;t)-\sum_{j=-k,\,\, j\neq\ell}^k G_j(n;t)P_{n+j}(x;t),$$
then by orthogonality and by the fact that $M^\dagger(t)=M(t),$
\begin{equation*}
\begin{aligned}
    G_\ell(n;t)\mathcal{H}_{n+\ell}(t)
    =
    \left\langle G_{\ell}(n;t)P_{n+\ell}(x;t),P_{n+\ell}(x;t)\right\rangle_t
    &=
    \left\langle M(t)\cdot P_{n}(x;t),P_{n+\ell}(x;t)\right\rangle_t\\
    &=
    \left\langle P_{n}(x;t),M(t)\cdot P_{n+\ell}(x;t)\right\rangle_t\\
    &= \mathcal{H}_{n}(t)G_{-\ell}(n+\ell;t)^\ast.
\end{aligned}
\end{equation*}
This completes the proof of the proposition. 
\end{proof}

\begin{rmk}
Observe that equation \eqref{eq:WeakPearson-H} has the form of the so called weak Pearson equation, see \cite[Section 2.3]{EMR}.
\end{rmk}

\subsection{Toda-type equations}\label{subsec:Toda}

In this section, we obtain time evolution equations in the variable $t$ for the coefficients $G_j(n;t)$. In analogy with the scalar case, we call them Toda-type equations.

We recall the notation $\langle\cdot,\cdot\rangle_t$ for the inner product introduced in \eqref{eq:MVOPs_t}, and we start with two auxiliary identities, which will be needed for the proof of the main theorem.

\begin{lem}
The following relation holds true:
\begin{equation}\label{eq:diffHn}
\dot{\mathcal{H}}_n(t)
    =-G_0(n;t)\mathcal{H}_n(t)
    =-\mathcal{H}_n(t)G_0(n;t)^\ast.
\end{equation}
\end{lem}

\begin{proof}
If we differentiate $\mathcal{H}_n(t)$ with respect to $t$, we obtain, by using \eqref{eq:MVOPs_t}, orthogonality and monicity of $P_n(x;t)$, that
\begin{equation*}
 \frac{d}{dt}\mathcal{H}_n(t)
 =
 -\langle P_n(x;t)\cdot\Lambda(x),P_n(x;t)\rangle_t
 =
 -\langle M(t)\cdot P_n(x;t),P_n(x;t)\rangle_t.
 \end{equation*}
 By orthogonality, we get
 \begin{equation*}
 \frac{d}{dt}\mathcal{H}_n(t)
 =
 -G_0(n;t)\mathcal{H}_n(t).
 \end{equation*}
The second equality in \eqref{eq:diffHn} follows directly from Proposition 4 with $\ell=0$.
\end{proof}

\begin{lem}
For $m=-k,\ldots,k$, we have the identity
\begin{equation}
G_m(n;t)\mathcal{H}_{n+m}(t)
=
\langle M(t)\cdot P_n(x;t),P_{n+m}(x,t)\rangle_t
=
\langle P_n(x;t),P_{n+m}(x,t)\cdot\Lambda(x)\rangle_t
\end{equation}
\end{lem}
\begin{proof}
We observe that, for $m=-k,\ldots,k$, using orthogonality,
\[
\begin{aligned}
\langle M(t)\cdot P_n(x;t),P_{n+m}(x,t)\rangle_t
&=
\sum_{j=-k}^k 
\langle G_j(t)P_{n+j}(x;t),P_{n+m}(x,t)\rangle_t\\
&=
\langle G_m(n;t)P_{n+m}(x;t),P_{n+m}(x,t)\rangle_t
=G_m(n;t)\mathcal{H}_{n+m}(t).
\end{aligned}
\]
The second equality follows from the construction of the operators $\Lambda(x)$ and $M(t)$ and the fact that $\Lambda(x)=\Lambda(x)^\dagger$.
\end{proof}

The main result in this section is presented in the next Theorem, which gives time evolution equations for the coefficients $G_m(n;t)$ of the difference operator $M(t)$ in terms of $t$. We note that we can have positive and negative values of $m$, and these two cases need to be separated when working out the time evolution equations.
\begin{theorem}
\label{thm:toda-type-eq}
The coefficients $G_m(n;t)$ satisfy the following time evolution equations in $t$:
\begin{align*}
\dot{G}_m(n;t)&=\sum_{j=-k}^{m}G_j(n;t)G_{m-j}(n+j;t) 
- \sum_{j=0}^{k+m} G_j(n;t)G_{m-j}(n+j;t), \qquad m=-k,\ldots,-1, \\
\dot{G}_m(n;t)&=\sum_{j=-k+m}^{-1}G_j(n;t)G_{m-j}(n+j;t)  - \sum_{j=m+1}^{k}G_j(n;t)G_{m-j}(n+j;t), \qquad m=0,\ldots,k.
\end{align*}
\end{theorem}
\begin{proof}
In order to derive the time evolution equations for $G_m(n;t)$, we note that for $m=-k,\ldots,k$, we have
\begin{equation}\label{eq:ddtGm}
\begin{aligned}
\frac{d}{dt}\left(G_m(n;t)\mathcal{H}_{n+m}(t)\right)
&=
\frac{d}{dt}
\langle P_n(x;t),P_{n+m}(x;t)\cdot\Lambda(x)\rangle_t=I_1(t)+I_2(t)+I_3(t),
\end{aligned}
\end{equation}
where
\begin{align*}
I_1(t) &= \langle \dot{P}_n(x;t),P_{n+m}(x,t)\cdot\Lambda(x)\rangle_t, \qquad 
I_2(t) = -\langle P_n(x;t)\cdot\Lambda(x),P_{n+m}(x,t)\cdot\Lambda(x)\rangle_t, \\
I_3(t) &=  \langle P_n(x;t),\dot{P}_{n+m}(x,t)\cdot\Lambda(x)\rangle_t.
\end{align*}
The term $I_2(t)$ is given by
\[
\begin{aligned}
I_2(t)
&=
-\langle P_n(x;t)\cdot\Lambda(x),P_{n+m}(x,t)\cdot\Lambda(x)\rangle_t\\
&=-\sum_{j,\ell=-k}^{k}\langle G_j(n;t)P_{n+j}(x;t), G_{\ell}(n+m;t)P_{n+m+\ell}(x;t) \rangle_t\\
&=
-\sum_{j=-k}^{k}
G_j(n;t)\mathcal{H}_{n+j}(t)G_{j-m}(n+m;t)^\ast\\
&=-\sum_{j=\max\{-k+m,-k\}}^{\min\{k+m,k\}}G_j(n;t)\mathcal{H}_{n+j}(t)G_{j-m}(n+m;t)^\ast.
\end{aligned}
\]
As in \eqref{eq:DifferenceOperator} we assume that $\mathcal{H}_n=0$ for all $n<0.$ Moreover, we observe that the index in the previous sum is restricted because otherwise the $G_j$ coefficients are zero.

Next we compute the term $I_1(t)$. For $n+m+j\leq n-1$ we get
\[
\begin{aligned}
    0&=
    \frac{d}{dt}\langle P_n(x;t),P_{n+m+j}(x;t)\rangle_t\\
    &=
    \langle \dot{P}_{n}(x;t),P_{n+m+j}(x;t)\rangle_t
    -
    \langle P_{n}(x;t)\cdot\Lambda(x),P_{n+m+j}(x;t)\rangle_t
    +
    \langle P_{n}(x;t),\dot{P}_{n+m+j}(x;t)\rangle_t,
\end{aligned}
\]
and the last term is equal to zero. Then 
\begin{equation*}
\begin{aligned}
\langle \dot{P}_n(x;t),P_{n+m+j}(x;t)\rangle_t
&=
\langle P_n(x;t)\cdot\Lambda(x),P_{n+m+j}(x;t)\rangle_t\\
&=
\langle P_n(x;t),M(t)\cdot P_{n+m+j}(x;t)\rangle_t
=
\mathcal{H}_n(t) G_{-m-j}(n+m+j;t)^\ast.
\end{aligned}
\end{equation*}
From the above equation we get
\begin{align*}
I_1(t)
=
\langle \dot{P}_n(x;t),P_{n+m}(x;t)\cdot\Lambda(x)\rangle_t
&=
\sum_{j=-k}^{-m-1}\langle \dot{P}_{n}(x;t),P_{n+m+j}(x;t)\rangle_t G_j(n+m;t)^\ast\\
&=\sum_{j=-k}^{-m-1} \mathcal{H}_n(t) G_{-m-j}(n+m+j;t)^\ast G_j(n+m;t)^\ast.
\end{align*}
Applying \eqref{eq:WeakPearson-H}, shifting the index of summation and restricting the index again, we obtain
\[
\begin{aligned}
I_1(t)
&= \sum_{j=-k}^{-m-1} G_{m+j}(n;t)\mathcal{H}_{n+m+j}(t) G_j(n+m;t)^\ast = \sum_{j=m-k}^{-1} G_{j}(n;t)\mathcal{H}_{n+j}G_{j-m}(n+m;t)^\ast\\
&=\sum_{j=\max\{-k+m,-k\}}^{-1}G_j(n;t)\mathcal{H}_{n+j}G_{j-m}(n+m;t)^\ast,
\end{aligned}
\]
By doing a similar calculation we get 
\begin{equation*}
    I_3(t)=\sum_{j=-k+m}^{m-1} G_j(n;t)\mathcal{H}_{n+j}(t)G_{j-m}(n+m;t)^\ast=\sum_{j=\max\{-k+m,-k\}}^{m-1}G_j(n;t)\mathcal{H}_{n+j}(t)G_{j-m}(n+m;t)^\ast.
\end{equation*}
After cancellation of terms in the sum $I_1(t)+I_2(t)+I_3(t)$, the right hand side of \eqref{eq:ddtGm} is equal to
$$
\sum_{j=\max\{-k+m,-k\}}^{\min\{-1,m-1\}}G_j(n;t)\mathcal{H}_{n+j}(t)G_{j-m}(n+m;t)^\ast
- \sum_{j=\max\{0,m\}}^{\min\{k+m,k\}}G_j(n;t)\mathcal{H}_{n+j}(t)G_{j-m}(n+m;t)^\ast.
$$
Applying \eqref{eq:WeakPearson-H} again, this is equal to
$$
\sum_{j=\max\{-k+m,-k\}}^{\min\{-1,m-1\}}G_j(n;t)G_{m-j}(n+j;t)\mathcal{H}_{n+m}(t) - \sum_{j=\max\{0,m\}}^{\min\{k+m,k\}}G_j(n;t)G_{m-j}(n+j;t)\mathcal{H}_{n+m}(t).
$$
On the other hand, for  \eqref{eq:ddtGm} we also have
$$
\begin{aligned}
\frac{d}{dt}(G_m(n;t)\mathcal{H}_{n+m}(t))
&=\dot{G}_m(n;t)\mathcal{H}_{n+m}(t)+G_m(n;t)\dot{\mathcal{H}}_{n+m}(t)\\
&=\dot{G}_m(n;t)\mathcal{H}_{n+m}(t)-G_m(n;t)G_0(n+m;t)\mathcal{H}_{n+m}(t),
\end{aligned}
$$
using \eqref{eq:diffHn}. Finally, multiplying by $\mathcal{H}_{n+m}(t)^{-1}$ from the right and rearranging terms, we obtain
\begin{multline*}
\dot{G}_m(n;t)=\sum_{j=\max\{-k+m,-k\}}^{\min\{-1,m-1\}}G_j(n;t)G_{m-j}(n+j;t) 
- \sum_{j=\max\{0,m\}}^{\min\{k+m,k\}}G_j(n;t)G_{m-j}(n+j;t)(t)\\+G_m(n;t)G_0(n+m;t).
\end{multline*}
This concludes the proof of the theorem. 
\end{proof}

Next we present the cases $k=1$ and $k=2$ explicitly.

\begin{corollary}[Toda-type lattice equations]
\label{cor:toda}
The case $k=1,$ i.e., $M$ of order three gives the Toda-type equations: 
\begin{align*}
   \dot{G}_0(n;t)&=G_{-1}(n;t)G_{1}(n-1;t)-G_1(n;t)G_{-1}(n+1;t), \\
   \dot{G}_{-1}(n;t)&=G_{-1}(n;t)G_0(n-1;t)-G_0(n;t)G_{-1}(n;t),\\
   \dot{G}_{1}(n;t)&=0.
\end{align*}
\end{corollary}

As an example, we consider $\Lambda(x)=x,$ in such case the difference operator is the one corresponding to the three-term recurrence relation: 
$$\mathcal{L}(t)=\delta + B(n;t) +C(n;t)\delta^{-1},$$
i.e., $G_1(n;t)=1,\,\, G_0(n;t)=B(n;t),\,\, G_{-1}(n;t)=C(n;t),$
and we get
\begin{align*}
\dot{B}(n;t)&=C(n;t)-C(n+1;t),\\
\dot{C}(n;t)&=C(n;t)B(n-1;t)-B(n;t)C(n;t).
\end{align*}
We recover the standard (non-Abelian) Toda lattice equations for the recurrence coefficients. For more general $\Lambda(x)$, we obtain coefficients $G_0(n;t)$ and $G_{\pm 1}(n;t)$ that can be written in terms of $B(n;t)$ and $C(n;t)$.
\begin{corollary}
The case $k=2,$ i.e., $M$ of order five gives the equations 
\begin{align*}
   \dot{G}_0(n;t)&=G_{-2}(n;t)G_{2}(n-2;t)+G_{-1}(n;t)G_{1}(n-1;t) -G_1(n;t)G_{-1}(n+1;t)\\
   &-G_{2}(n;t)G_{-2}(n+2;t),\\
    \dot{G}_{-1}(n;t)&=G_{-2}(n;t)G_1(n-2;t)-G_0(n;t)G_{-1}(n;t)-G_{1}(n;t)G_{-2}(n+1;t)\\
    &+G_{-1}(n;t)G_0(n-1;t),\\
    \dot{G}_{1}(n;t)&=G_{-1}(n;t)G_2(n-1;t)-G_{2}(n;t)G_{-1}(n+2;t),\\
    \dot{G}_{-2}(n;t)&=G_{-2}(n;t)G_0(n-2;t)-G_{0}(n;t)G_{-2}(n;t),\\
    \dot{G}_{2}(n;t)&=0.
\end{align*}
As before, in this case the terms $G_{m}(n;t)$ can be written in terms of the coefficients $B(n;t)$ and $C(n;t)$ of the three-term recurrence relation for MVOPs given in \eqref{eq:three-term}, but we omit this calculation.

\end{corollary}

\subsection{Lax pair}\label{subsec:Lax}
The time evolution equations for the coefficients $G_m(n;t)$ that we have obtained can also be written in the form of a Lax pair. Let us consider the block infinite matrices $L,$ $L^{+},$ where $i,j\in\mathbb{N}_0,$
$$L_{i,j}= \left\{ \begin{array}{lcc}
             0 &  |i-j|>k\\
             G_{j-i}(i;t) & |i-j|\leq k\\
             \end{array}
   \right. \qquad \qquad\qquad L^{+}_{i,j}=\left\{ \begin{array}{lcc}
             L_{i,j} &  i\leq j\\
            0 & i>j\\
             \end{array}
   \right.$$
\begin{theorem}[Lax pair] The following relation holds
$$\dot{L}=[L,L^+].$$
where $[L,L^+]=LL^+-L^+L.$
\end{theorem}
\begin{proof}
Let us compute 
\begin{multline*}
(LL^{+})_{n,n+m}=\sum_{s=0}^\infty L_{n,s}L^{+}_{s,n+m}=\sum_{r=-n}^\infty L_{n,r+n}L^{+}_{r+n,n+m}=\sum_{r=\max\{-n,-k\}}^{\min\{k,m\}}L_{n,r+n}L^{+}_{r+n,n+m}\\
=\sum_{r=\max\{-n,-k\}}^{\min\{k,m\}}G_r(n;t)G_{m-r}(r+n;t) = \sum_{r=\max\{-n,-k,m-k\}}^{\min\{k,k+m,m\}} G_r(n;t)G_{m-r}(r+n;t).
\end{multline*}
Similarly we get 
$$(L^{+}L)_{n,n+m}=\sum_{r=0}^{\min\{k,k+m\}} G_r(n;t)G_{m-r}(r+n;t).$$
Computing the bracket we get the desired result:
\begin{multline*}
[L,L^{+}]_{n,n+m}=(LL^{+})_{n,n+m}-(L^{+}L)_{n,n+m}\\
=\sum_{r=\max\{-n,-k,m-k\}}^{\min\{k,k+m,m\}} G_r(n;t)G_{m-r}(r+n;t)-\sum_{r=0}^{\min\{k,k+m\}} G_r(n;t)G_{m-r}(r+n;t).
\end{multline*}
Therefore we have
\begin{align*}
[L,L^{+}]_{n,n+m}&=\sum_{r=-k}^{m} G_r(n;t)G_{m-r}(r+n;t)-\sum_{r=0}^{k+m} G_r(n;t)G_{m-r}(r+n;t), \quad
m=-k,\ldots,-1, \\
[L,L^{+}]_{n,n+m}&=\sum_{r=m-k}^{m} G_r(n;t)G_{m-r}(r+n;t)-\sum_{r=0}^{k} G_r(n;t)G_{m-r}(r+n;t), \quad
m=0,\ldots,k.
\end{align*}
Here we assume that the sequence $G_j(n;t)=0$ for negative values of $n$. It now follows that $[L,L^{+}]_{n,n+m}$ coincides with the right hand side of the time evolution equations in Theorem \ref{thm:toda-type-eq}. This completes the proof
of the theorem.
\end{proof}

We note that this result is a block version of the scalar case given in \eqref{eq:JBtilde} and it also generalises several known cases in the literature: we can consider deformation with any scalar monomial $\Lambda(x)=x^k$, for $k\geq 1$, which leads to the Toda hierarchy as explained in \cite[Section 2.3]{WalterOPsToda}, and also with $\Lambda(x)$ formed of linear combinations of different powers of $x$. Moreover, the result in this section includes the case of matrix-valued $\Lambda(x)\in\mathcal{S}(W)$, see 
\eqref{eq:defS}.

\section{Polynomial Toda equations}\label{sec:polToda}

In the previous section, we considered the general case of a difference operator $M(t)$ with $2k+1$ terms, see \eqref{eq:M-2k+1}, which comes from a polynomial $\Lambda(x)$ of $k$ degree of the form given in \eqref{eq:Lambdak}. We observe that such form can also be obtained as a composition $v(\Lambda(x);t)$, with $\Lambda(x)$ of degree one and $v(x;t)$ a polynomial of degree $k$ in $x$. In this case, apart from $M(t)$, there exists another difference operator $M_{\Lambda}(t)$, which is associated with $\Lambda(x)$ and has three terms. In this section we study time evolution of the coefficients of this difference operator, with respect to an extra real parameter $t$.

Let $\Lambda(x)\in\mathcal{S}(W)$ be     a polynomial of degree one and $v$ a polynomial of degree $k$ in $x$ which depends smoothly in a parameter $t$ such that 
$$H(x;t)=e^{-v(\Lambda(x);t)}W(x),$$
has finite moments of all order. Let $P^H_n(x;t)$ be the monic orthogonal polynomials with respect to $H$. We observe that both $\Lambda(x), v(\Lambda(x);t)\in\mathcal{S}(H(x;t)),$ and then 
\[
P_n^H \cdot\Lambda(x)=M_\Lambda(t)\cdot P_n^H
\quad 
\Longrightarrow 
\quad
P_n^H \cdot v(\Lambda(x);t)=v(M_{\Lambda};t)\cdot P_n^H,
\]
where 
\begin{equation}
\label{eq:Mlambda}
M_\Lambda(t)=K_1(n;t)\delta+K_0(n;t)+K_{-1}(n;t)\delta^{-1},
\end{equation}
and $v(M_{\Lambda};t)$ has $2k+1$ terms, which can be written as combinations of $K_{\pm 1}(n;t)$ and $K_0(n;t)$. We observe that the coefficient $K_1(n;t)$ is in fact  independent of $t$ and $n$ by comparing leading coefficients as in the proof of Lemma $\ref{lem:3}$, but we keep the notation for consistency.

Following \cite[Remark 3]{DER}, given the difference operator $M_\Lambda$ corresponding to the polynomial of degree one $\Lambda$, and any polynomial $q$, 
we denote by $q(M_\Lambda;t)_{\ell}(n;t)$ the coefficient of $q(M_\Lambda;t)$ of order $\ell$ in $\delta$:
\begin{equation}
\label{rmk:notation}
q(M_\Lambda;t)=\sum_{j=-\deg q}^{\deg q} q(M_\Lambda;t)_{j}(n;t) \, \delta^j.
\end{equation}
As in \cite{DER} the coefficient $q(M_\Lambda;t)_{\ell}(n;t)$ can be carried out following the scheme shown in Figure \ref{fig:pLj}: $q(M_\Lambda;t)_{\ell}(n;t)$ is equal to the sum over all possible paths from $P^H_n(x;t)$ to $P^H_{n+\ell}(x;t)$ in $\deg q$ steps, where in each path we multiply the coefficients corresponding to each arrow.
\begin{figure}[!h]
\begin{center}
\begin{tikzcd}[row sep=normal, column sep=50pt]
\cdots	\arrow[shift left]{r}{K_{-1}(n+2;t)}
& P^H_{n+1}(x;t)  	\arrow[shift left]{r}{K_{-1}(n+1;t)}
		\arrow[shift left]{l}{K_{1}(n+1;t)}
	         \arrow[out=120, in=60, loop]{r}{K_{0}(n+1;t)}
& P^H_{n}(x;t)
		\arrow[shift left]{l}{K_{1}(n;t)}
		\arrow[shift left]{r}{K_{-1}(n;t)}
	         \arrow[out=120, in=60, loop]{r}{K_{0}(n;t)}
& P^H_{n-1}(x;t) \arrow[shift left]{l}{K_{1}(n-1;t)}
		\arrow[shift left]{r}{K_{-1}(n-1;t)}
	         \arrow[out=120, in=60, loop]{r}{K_{0}(n-1;t)}
& \cdots 	 \arrow[shift left]{l}{K_{1}(n-2;t)}
\end{tikzcd}
\caption{Scheme for the calculation of $q(M_\Lambda;t)_{\ell}(n;t)$.}
\label{fig:pLj}
\end{center}
\end{figure}
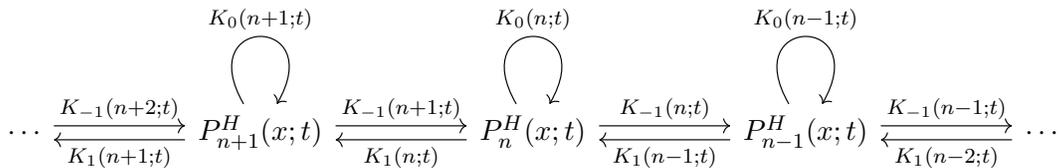

Before we present the results on deformation of the coefficients $K_0(n;t), K_{\pm 1}(n;t)$, we need a lemma on deformation of the orthogonal polynomials themselves. 

\begin{lem}\label{lem12} We have the following $t$ evolution of the orthogonal polynomials:
$$\dot{P}_n^H(x;t)=\sum_{m=0}^{n-1}\dot{v}(M_\Lambda;t)_{m-n}(n;t)P^H_m(x;t)$$
\end{lem}
\begin{proof}
Observe that if $n>m,$
\begin{multline*}
0=\frac{d}{dt}\langle P^H_n(x;t), P^H_m(x;t)\rangle_t =
\langle \dot{P}^H_n(x;t), P^H_m(x;t)\rangle_t - \langle P^H_n(x;t)\cdot\dot{v}(\Lambda(x);t), P^H_m(x;t)\rangle_t\\
+ \langle P^H_n(x;t), \dot{P}^H_m(x;t)\rangle_t.    
\end{multline*}
Since $n>m$ the last term of the previous equation is zero and
\[\langle \dot{P}^H_n(x;t), P^H_m(x;t)\rangle_t= \langle P^H_n(x;t)\cdot\dot{v}(\Lambda(x);t), P^H_m(x;t)\rangle_t.\]
Additionally $\dot{P}_n^H(x;t)$ is a polynomial of degree $n-1,$ then by orthogonality
\[\begin{aligned}
\dot{P}_n^H(x;t)&=\sum_{m=0}^{n-1}\langle\dot{P}_n^H(x;t), P^H_m(x;t)\rangle_t \mathcal{H}_m(t)^{-1}P^H_m(x;t)\\
&=\sum_{m=0}^{n-1}\langle P_n^H(x;t)\cdot\dot{v}(\Lambda(x);t), P^H_m(x;t)\rangle_t \mathcal{H}_m(t)^{-1}P^H_m(x;t)\\
&=\sum_{m=0}^{n-1}\langle \dot{v}(M_\Lambda;t)\cdot P_n^H(x;t), P^H_m(x;t)\rangle_t \mathcal{H}_m(t)^{-1}P^H_m(x;t)\\
&=\sum_{m=0}^{n-1} \dot{v}(M_\Lambda;t)_{m-n}(n;t)P^H_m(x;t).
\end{aligned}
\]
This concludes the proof of the lemma.
\end{proof}

\begin{rmk}
In the case of the Toda lattice equations in Corollary \ref{cor:toda}, we have $v(x;t)=tx$. Hence $v(\Lambda(x);t)=t\Lambda(x)$, and $\dot{v}=\Lambda(x)$, then $k=1$, and the only non-zero term corresponds to $m=n-1$, which leads to
$$\dot{P}_n^H(x;t)=\sum_{m=0}^{n-1} K_{m-n}(n;t)P^H_m(x;t)=K_{-1}(n;t)P^H_{n-1}(x;t).$$
If additionally $\Lambda(x)=x$, then $M_\Lambda$ is the operator related to the three-term recurrence relation and
$$\dot{P}_n^H(x;t)=C(n;t)P^H_{n-1}(x;t).$$

This is compatible with the well known result of Toda deformation in the scalar case, see for instance \cite[Lemma 2.1]{WalterOPsToda}.
\end{rmk}

As in the previous section, we present time evolution equations for the coefficients $K_0(n;t)$ and $K_{\pm 1}(n;t)$ of the difference operator.
\begin{theorem}
\label{thm:K-dotv}
The following time evolution equations hold:
\[\begin{aligned}
\dot{K}_0(n;t)&=\dot v(M_\Lambda;t)_{-1}(n;t)K_1(n-1;t)-K_1(n;t)\dot{v}(M_\Lambda;t)_{-1}(n+1;t)\\
  \dot{K}_{-1}(n;t)&=\dot{v}(M_\Lambda;t)_{-2}(n;t)K_1(n-2;t)+\dot{v}(M_\Lambda;t)_{-1}(n;t)K_0(n-1;t)\\
  &\hspace{4cm}
   -K_1(n;t)\dot{v}(M_\Lambda;t)_{-2}(n+1;t) -K_0(n;t)\dot{v}(M_\Lambda;t)_{-1}(n;t)
\end{aligned}
\]
\end{theorem}
\begin{proof}
We take $\frac{d}{dt}$ on both sides of 
$M_\Lambda(t)\cdot P_n^H(x;t)=P_{n}^H(x;t)\cdot \Lambda(x),$
and then the right hand side is equal to
\begin{align*}
\dot{P}_n^H(x;t)\cdot\Lambda(x)&=\sum_{m=0}^{n-1}\dot{v}(M_\Lambda;t)_{m-n}(n;t)P^H_m(x;t)\cdot\Lambda(x)\\
    &=\sum_{m=0}^{n-1}\dot{v}(M_\Lambda;t)_{m-n}(n;t)(K_1(m;t)P^H_{m+1}(x;t)+K_0(m;t)P^H_m(x;t)\\
    &\hspace{8cm}+K_{-1}(m;t)P^H_{m-1}(x;t)).
\end{align*}
Using Lemma \ref{lem12}, the left hand side is equal to 
\begin{multline*}
K_1(n;t)\dot{P}^H_{n+1}(x;t)+\dot{K}_0(n;t)P^H_{n}(x;t)+K_0(n;t)\dot{P}^H_{n}(x;t)+\dot{K}_{-1}(n;t)P^H_{n-1}(x;t)+K_{-1}(n;t)\dot{P}^H_{n-1}(x;t)\\
=\dot{K}_0(n;t)P^H_{n}(x;t)+\dot{K}_{-1}(n;t)P^H_{n-1}(x;t)+K_1(n;t)\sum_{m=0}^{n}\dot{v}(M_\Lambda;t)_{m-n-1}(n+1;t)P^H_m(x;t)\\
+K_0(n;t)\sum_{m=0}^{n-1}\dot{v}(M_\Lambda;t)_{m-n}(n;t)P^H_m(x;t)+K_{-1}(n;t)\sum_{m=0}^{n-2}\dot{v}(M_\Lambda;t)_{m-n+1}(n-1;t)P^H_m(x;t).
\end{multline*}
Comparing coefficients of $P_n^H$  on both sides, we get
\[\dot{K}_0(n;t)=\dot v(M_\Lambda;t)_{-1}(n;t)K_1(n-1;t)-K_1(n;t)\dot{v}(M_\Lambda;t)_{-1}(n+1;t),\]
and by comparison of coefficients again, we arrive at
\begin{multline*}
    \dot{K}_{-1}(n;t)=\dot{v}(M_\Lambda;t)_{-2}(n;t)K_1(n-2;t)+\dot{v}(M_\Lambda;t)_{-1}(n;t)K_0(n-1;t)\\-K_1(n;t)\dot{v}(M_\Lambda;t)_{-2}(n+1;t)
    -K_0(n;t)\dot{v}(M_\Lambda;t)_{-1}(n;t).
\end{multline*}
This concludes the proof of the theorem.
\end{proof}


\begin{corollary}[Langmuir-type equations]
For $v(x;t)=tx^2$ we get
\begin{multline*}
\dot{K}_0(n;t)=(K_0(n;t)K_{-1}(n;t)+K_{-1}(n;t)K_0(n-1;t))K_1(n-1;t)\\-K_{1}(n;t)(K_0(n+1;t)K_{-1}(n+1;t)+K_{-1}(n+1;t)K_0(n;t)),
\end{multline*}
\begin{multline*}
\dot{K}_{-1}(n;t)=K_{-1}(n;t)K_{-1}(n-1;t)K_1(n-2;t)-K_1(n;t) K_{-1}(n+1;t)K_{-1}(n;t)\\
+K_{-1}(n;t)K_0(n-1;t)^2-K_0(n;t)^2K_{-1}(n;t).
\end{multline*}
\end{corollary}
\begin{proof}
The proof is a direct consequence of Theorem \ref{thm:K-dotv}. First we note that $v(\Lambda(x);t)=t\Lambda(x)^2$ and hence $\dot{v}(\Lambda(x);t)=\Lambda(x)^2$. From the expression of the corresponding $M_\Lambda$ in \eqref{eq:Mlambda}, we get
\begin{align*}
\dot{v}(M_\Lambda;t)&=\dot{v}(M_\Lambda;t)_{2}(n)\delta^2+\dot{v}(M_\Lambda;t)_{1}(n)\delta+\dot{v}(M_\Lambda;t)_0(n)+\dot{v}(M_\Lambda;t)_{-1}(n)\delta^{-1}+\dot{v}(M_\Lambda;t)_{-2}(n)\delta^{-2},
\end{align*}
where the coefficients on $\dot{v}(M_\Lambda;t)$ can be written in terms of $K_1,\,\,K_0,\,\,K_{-1}$ as:
\begin{align*}
    \dot{v}(M_\Lambda;t)_2(n)&=K_1(n;t)K_1(n+1;t),\\
    \dot{v}(M_\Lambda;t)_1(n)&=K_1(n;t)K_0(n+1;t)+K_0(n;t)K_1(n;t),\\
    \dot{v}(M_\Lambda;t)_0(n)&=K_1(n-1;t)K_{-1}(n+1;t)+K_0(n)^2+K_{-1}(n;t)K_1(n-1;t),\\
    \dot{v}(M_\Lambda;t)_{-1}(n)&=K_0(n;t)K_{-1}(n;t)+K_{-1}(n;t)K_0(n-1;t),\\
   \dot{v}(M_\Lambda;t)_{-2}(n)&=K_{-1}(n;t)K_{-1}(n-1;t).
\end{align*}Now the corollary follows from Theorem \ref{thm:K-dotv}.
\end{proof}

\section{Deformed Hermite type polynomials}\label{sec:Hermite}
In this section, we consider the $N\times N$ Hermite-type weight matrix $W$ supported on $\mathbb{R}$ studied in \cite{IKR2}:
\begin{equation} 
\label{eq:ourweight}
W^{(a)}(x) = e^{-x^2} e^{xA}e^{xA^\ast}, \qquad A_{j,k}
=
\begin{cases}
a_k & k = j-1, \\
0 & \textrm{otherwise}.
\end{cases}
\end{equation}
where $a$ is vector of parameters $a=(a_j)_{j=1}^{N-1}$. In this example the indices of the matrices are  labeled by $j,k \in \{1,\ldots N\}$. We denote by $(P_n^{(a)})_n$ the sequence of monic orthogonal polynomials with respect to $W^{(a)}$. The right Fourier algebra $\mathcal{F}_R(P^{(a)})$ contains four relevant differential operators, namely the operator of order zero given by the multiplication by $x$, and
$$D^{(a)}=\frac{d}{dx}+A, \qquad (D^{(a)})^\dagger=-D^{(a)}+2x, \qquad \mathcal{D}^{(a)} = \frac{d^2}{dx^2} + \frac{d}{dx} (2A-2x) + A^2-2J,$$
see \cite{DER}. We note that the expressions in this paper differ from those in \cite{DER} by a conjugation with a constant matrix $L(0)$. The operators $D^{(a)}$ and $(D^{(a)})^\dagger$ are mutually adjoint and $\mathcal{D}^{(a)}$ is a symmetric operator. The
Lie algebra $\mathfrak{g}$ generated by $D^{(a)}$, $(D^{(a)})^\dagger$, $\mathcal{D}^{(a)}$, $x$ and the identity matrix is a four dimensional Lie algebra which is the harmonic oscillator algebra \cite[Proposition 4]{DER}. The Casimir operator for this Lie algebra is the polynomial of degree one
\begin{equation}
\label{eq:Casimir-Hermite}
\mathcal{C}^{(a)}(x) = \mathcal{D}^{(a)}-\frac12 (D^{(a)})^\dagger D^{(a)} =  J - xA,\qquad J_{ii} = i, \qquad J_{ij}=0, \,\, i\neq j.
\end{equation}
The fundamental property of the Casimir operator $\mathcal{C}$ is that it commutes with $D^{(a)}$, $(D^{(a)})^\dagger$, $\mathcal{D}^{(a)}$ and $x$. The recurrence relation for the monic matrix valued orthogonal polynomials is given by
$$xP^{(a)}_n(x) = P^{(a)}_{n+1}(x) + B^{(a)}(n) P^{(a)}_{n}(x) + C^{(a)}(n) P^{(a)}_{n-1}(x).$$
If the parameters $a_k$ satisfy a special relation, the coefficients $B^{(a)}(n)$ and $C^{(a)}(n)$ have simple expressions \cite[Proposition 3.16]{IKR2}. 

\subsection{The deformed weight}
From \eqref{eq:Casimir-Hermite} we obtain that $\mathcal{C}^{(a)}(x)$ is a symmetric operator and, hence $\mathcal{C}^{(a)}(x) \in \mathcal{S}(W)$. We can therefore consider a deformation of the weight $W$ with the Casimir operator
\begin{equation}\label{eq:WtCasimir}
    W(x;t)=e^{-t\mathcal{C}^{(a)}(x)}W^{(a)}(x).
\end{equation}

From \cite[(3.6)]{IKR}, we obtain $e^{xA}J e^{-xA}=\mathcal{C}^{(a)}(x)$ and then $e^{xA}e^{-tJ}e^{-xA} = e^{-t\mathcal{C}^{(a)}(x)}$. From this we get the following expression of the deformed weight:
\begin{equation}
    \label{eq:Wxt-tinJ}
    W(x;t) = e^{-x^2}e^{xA}e^{-tJ}e^{xA^\ast}.
\end{equation}
In \eqref{eq:Wxt-tinJ}, the $t$--dependence is only in the diagonal constant factor $e^{-tJ}$. On the other hand, a straightforward calculation gives another factorization of $W(x;t)$:
\begin{equation}
\label{eq:Wxt-W}
W(x;t)=e^{-\frac{t}{2}J} W^{(e^\frac{t}{2}a)}(x)e^{-\frac{t}{2}J}.
\end{equation}
From the expression \eqref{eq:Wxt-W}, we obtain the deformed monic orthogonal polynomials with respect to $W(x;t)$ in terms of the monic orthogonal polynomials with respect to the weight $W^{(e^{\frac{t}{2}}a)}(x)$:
\begin{equation}
\label{eq:Pnt-in-term-Pn}
P_n(x;t) = e^{-\frac{t}{2}J} P_n^{(e^{\frac{t}{2}}a)}(x)e^{\frac{t}{2}J}.
\end{equation}

\subsection{Deformed Fourier algebras}
As a consequence the results in the previous sections, we describe the deformed Fourier algebras $\mathcal{F}_L(P;t)$, $\mathcal{F}_R(P;t)$.
\begin{theorem}
\label{eq:diagrama}
The following diagram is commutative
$$
\begin{tikzcd}
	\mathcal{F}_L(P^{(e^{\frac{t}{2}}a)}) \arrow[r, "\varphi"] \arrow[d, "\tau"]
	& \mathcal{F}_R(P^{(e^{\frac{t}{2}}a)}) \arrow[d, "\sigma"] \\
	\mathcal{F}_L(P;t)\arrow[r, "\varphi_t"] 
	& \mathcal{F}_R(P;t)
\end{tikzcd}
$$
where the maps
$\sigma$ and $\tau$
are defined by $$\sigma(E)=e^{-\frac{t}{2}J}Ee^{\frac{t}{2}J},\qquad \tau(M)=e^{-\frac{t}{2}J}Me^{\frac{t}{2}J}$$
and are algebra isomorphisms.
\end{theorem}
\begin{proof}
Let $M\in\mathcal{F}_L(P^{(e^{\frac{t}{2}}a)})$ and $E=\varphi(M)$ i.e., $P^{(e^{\frac{t}{2}}a)}_n(x)\cdot E= M\cdot P^{(e^{\frac{t}{2}}a)}_n(x).$ By \eqref{eq:Pnt-in-term-Pn} we get
$$P_n(x;t)\cdot e^{-\frac{t}{2}J}Ee^{\frac{t}{2}J}= e^{-\frac{t}{2}J}Me^{\frac{t}{2}J}\cdot P_n(x;t),$$
then $\varphi_t(\tau(M))=\sigma(\varphi(M)),$ and the diagram is commutative. The maps $\sigma$ and $\tau$ are algebra homomorphism, and their inverses are
$\sigma^{-1}(D)=e^{\frac{t}{2}J}De^{-\frac{t}{2}J}$ and $\tau^{-1}(M)=e^{\frac{t}{2}J}Me^{-\frac{t}{2}J}.$ This completes the proof of the theorem.
\end{proof}
\begin{corollary}
The deformed recurrence coefficients are given by
$$B(n;t) = e^{-\frac{t}{2}J}B^{(e^\frac{t}{2}a)}(n) e^{\frac{t}{2}J},\qquad  C(n;t) = e^{-\frac{t}{2}J}C^{(e^\frac{t}{2}a)}(n) e^{\frac{t}{2}J},$$
where $B^{(e^\frac{t}{2}a)}(n)$ and $C^{(e^\frac{t}{2}a)}(n)$ are the recurrence coefficients corresponding to $W^{(e^{\frac{t}{2}}a)}(x)$.
\end{corollary}

The next proposition shows that the maps $\tau$ and $\sigma$ preserve the $\dagger$-operations.
\begin{prop}
\label{prop:star-invariance}
Let $E\in\mathcal{F}_R(P^{(e^\frac{t}{2}a)})$ and let $E^\dagger$ be its adjoint. Let $M=\varphi^{-1}(E).$  Then 
$$\sigma(E^\dagger)=(\sigma(E))^\dagger,\qquad \tau(M^\dagger)=(\tau(M))^\dagger.$$
\end{prop}
\begin{proof}
The proof is a consequence of the factorization of the weight \eqref{eq:Wxt-W} and the expression of the monic orthogonal polynomials \eqref{eq:Pnt-in-term-Pn}, for
\begin{align*}
\langle P_n(x;t)\cdot\sigma(E),P_m(x;t)\rangle_t=e^{-\frac{t}{2}J}\langle P_n^{(e^\frac{t}{2}a)}\cdot E, P_m^{(e^\frac{t}{2}a)}\rangle e^{-\frac{t}{2}J}&=e^{-\frac{t}{2}J}\langle P_n^{(e^\frac{t}{2}a)}, P_m^{(e^\frac{t}{2}a)}\cdot E^\dagger\rangle e^{-\frac{t}{2}J}\\
&=\langle P_n(x;t),P_m(x;t)\cdot\sigma(E^\dagger)\rangle_t.   
\end{align*}
Similarly we get
\begin{align*}
\langle \tau(M)\cdot P_n(x;t)
,P_m(x;t)\rangle_t=e^{-\frac{t}{2}J}\langle M\cdot P_n^{(e^\frac{t}{2}a)}, P_m^{(e^\frac{t}{2}a)}\rangle e^{-\frac{t}{2}J}&=e^{-\frac{t}{2}J}\langle P_n^{(e^\frac{t}{2}a)}, M^\dagger \cdot P_m^{(e^\frac{t}{2}a)}\rangle e^{-\frac{t}{2}J}\\
&=\langle P_n(x;t),\tau(M^\dagger)\cdot P_m(x;t)\rangle_t.
\end{align*}
This completes the proof of the proposition
\end{proof}
\begin{corollary}
The following hold
$$\sigma(D^{(e^\frac{t}{2}a)})=D^{(a)},\qquad \sigma((D^{(e^\frac{t}{2}a)})^\dagger)=(D^{(a)})^\dagger, \qquad \sigma(\mathcal{D}^{(e^\frac{t}{2}a)})=\mathcal{D}^{(a)}, \qquad \sigma(\mathcal{C}^{(e^\frac{t}{2}a)})=\mathcal{C}^{(a)}.$$
The operators $D^{(a)}$, $(D^{(a)})^\dagger$, $\mathcal{D}^{(a)}$ and $\mathcal{C}^{(a)}$ belong to $\mathcal{F}_R(P;t)$ for all $t\geq0$. Moreover, $\mathcal{D}^{(a)}$ has the monic polynomials $P_n(x;t)$ as eigenfunctions.
\end{corollary}
\begin{proof}
The proof follows from the previous proposition.
\end{proof}
One of the consequences of Theorem \ref{eq:diagrama} and Proposition \ref{prop:star-invariance} is that,  despite time evolution, $D^{(a)}$ and $(D^{(a)})^\dagger$ are each other's adjoints and $\mathcal{D}^{(a)}$ is a symmetric differential operator. Moreover
the harmonic oscillator Lie algebra  $\mathfrak{g}$ is contained in $\mathcal{F}_R(P;t)$ for all $t\geq 0$. 

Finally, we observe that the Casimir operator $\mathcal{C}^{(a)}$ acts on the deformed polynomials $P_n(x,t)$ by a three-term difference operator given by
\begin{equation}
\label{eq:ActionCasimirHermite}
G_1(n;t) P_{n+1}(x;t) + G_0(n;t) P_n(x;t) + G_{-1}(n;t) P_{n-1}(x;t) = P_n(x;t)\cdot \mathcal{C}^{(a)}(x).
\end{equation}
From \cite[Lemma 4]{DER} we obtain the following expression for the coefficients $G_1$, $G_0$, $G_{-1}$:
$$G_1(n;t) = -A, \quad G_0(n;t) = n+J-2C(n;t)-AB(n;t), \quad G_{-1}(n;t) = C(n;t)A-2C(n;t)B(n-1;t).$$
By Theorem \ref{thm:toda-type-eq}, the coefficients $G_0$ and $G_{-1}$ satisfy the Toda type equations. Also, note that the coefficient $G_1$ is independent of $n$ and $t$, consistently with Lemma \ref{lem:3}.

\subsection{Example: the deformed $2\times 2$ weight}
If we let $N=2$ in \eqref{eq:Wxt-tinJ}, with $A$ given by \eqref{eq:ourweight} and $a_1=a$, the deformed weight matrix $W$ is the $2\times 2$ matrix
$$W(x;t) = 
e^{-x^2-t}
\begin{pmatrix}
1 & xa\\
xa& x^2a^2+e^{-t}
\end{pmatrix}  $$
The monic orthogonal polynomials $P_n(x;t)$ can be written as a matrix linear combination of scalar Hermite polynomial as
$$2^n P_{n}(x;t)= H_{n}(x) - n a\begin{pmatrix}
0 & \frac{2}{na^2+2e^t} \\
1 & 0 
\end{pmatrix}H_{n-1}(x) + n(n-1)\begin{pmatrix}
\frac{2a^2}{na^2+2e^t} & 0\\
0 & 0 
\end{pmatrix} H_{n-2}(x).$$
The case $t=0$ is given in \cite[Theorem 5.1]{DG2}, up to a conjugation of the matrix $A$.

In the $2\times2$ case, the coefficients $G_1$, $G_0$ and $G_{-1}$ associated to the Casimir operator in \eqref{eq:ActionCasimirHermite} have the simple expressions 
$$
G_1(n;t)=-A,\qquad G_0(n;t) = \begin{pmatrix}
2\frac{na^2+e^t}{na^2+2e^t} & 0\\
0 & \frac{(n+1)a^2+4e^t}{(n+1)a^2+2e^t}
\end{pmatrix}, \qquad G_{-1}(n;t) = \begin{pmatrix}
 0 & -\frac{2nae^t}{(na^2+2e^t)^2}\\
0 & 0
\end{pmatrix}.
$$
The Toda type equations in Theorem \ref{thm:toda-type-eq} can be directly verified from these explicit expressions. 

\section{Acknowledgements}
The authors are grateful to the Radboud Summer School Orthogonal Polynomials, Special Functions and their Applications 2022, where this project started. The authors would also like to thank Erik Koelink for useful discussions. 

The work of Luc\'ia Morey and Pablo Rom\'an was supported by SeCyTUNC.

Alfredo Dea\~{n}o acknowledges financial support from Universidad Carlos III de Madrid (I Convocatoria para la Recualificaci\'on del Profesorado Universitario), from Direcci\'on General de Investigaci\'on e Innovaci\'on, Consejer\'ia de Educaci\'on e Investigaci\'on of Comunidad de Madrid (Spain), and Universidad de Alcal\'a under grant CM/JIN/2021-014.
Research supported by Grant PID2021-123969NB-I00, funded by MCIN/AEI/ 10.13039/501100011033, and by grant PID2021-122154NB-I00 from Spanish Agencia Estatal de Investigaci\'on.

\end{document}